\documentclass[a4paper,11pt,reqno]{amsart}
\RequirePackage[table]	{xcolor}			
\definecolor{green1}{RGB}{0,107,28}
\definecolor{green2}{RGB}{17,85,35}
\definecolor{green3}{RGB}{0,77,20}
\definecolor{green4}{RGB}{33,166,68}
\definecolor{green5}{RGB}{60,166,88}

\definecolor{blue1}{RGB}{3,56,91}
\definecolor{blue2}{RGB}{16,51,73}
\definecolor{blue3}{RGB}{1,40,66}
\definecolor{blue4}{RGB}{35,109,157}
\definecolor{blue5}{RGB}{59,118,157}
\RequirePackage[
	final,							
	colorlinks=true,
	urlcolor=blue3,						
	linkcolor=blue3,					
	citecolor=green4,					
	hypertexnames=false,					
	hyperfootnotes=false,					
	]{hyperref}						

\usepackage{a4wide}
\usepackage{latexsym}
\usepackage{amsopn,amscd}
\usepackage{amsfonts}
\usepackage{amssymb}
\usepackage{amsthm}
\usepackage{amsmath}
\usepackage{todonotes}

\usepackage[all]{xy}
\usepackage{scalefnt}

\usepackage{hyperref}
\usepackage{enumitem}
\setitemize{leftmargin=*}
\setenumerate{leftmargin=0.45cm}

\newtheorem{thm}{Theorem}[section]

\newtheorem{prop}[thm]{Proposition}

\newtheorem{examp}[thm]{Example}
\newtheorem{rema}[thm]{Remark}

\def\ZZ{\mathbb Z}
\def\gg{\mathfrak g}
\def\kk{\mathfrak k}
\def\HH{\mathrm H}
\def\TTT{\mathrm T}
\def\TTTc{\mathrm T^{\mathrm c}}
\def\vv{\mathbf v}
\def\Sigm{\mathrm S}
\def\Sigmc{\mathrm S^{\mathrm c}}
\def\fiel{\mathbb F}
\def\Lc{\Lambda^{\mathrm c}}

\def\LLL{L}

\def\CC{\mathbb C}
\def\AAA{\mathcal A}
\def\Hom{\mathrm{Hom}}
\def\Coder{\mathrm{Coder}}
\def\Der{\mathrm {Der}}
\def\RR{\mathbb R}
\def\bra{[ \,\cdot\, , \,\cdot\,]}
\def\DDD{\mathcal D}
\def\KKK{\mathcal K}
\def\MMM{\mathcal M}
\def\VVV{\mathcal V}
\def\WWW{\mathcal W}
\def\OOO{\mathcal O}
\def\HHH{\mathcal H}
\def\CCC{\mathcal C}
\def\LLL{\mathcal L}

\def\ppartial{\mathcal D}

\def\Nsddata#1#2#3#4#5{
   (#4
     \begin{CD}
      \null @>#2>> \null\\[-3.2ex]
      \null @<<#3< \null
     \end{CD}
    #1, #5)
}
\long
\def\MSC#1\EndMSC{\def\arg{#1}\ifx\arg\empty\relax\else
      {\par\narrower\noindent
      2010 Mathematics Subject Classification. #1\par}\fi}

\long
\def\KEY#1\EndKEY{\def\arg{#1}\ifx\arg\empty\relax\else
    {\par\narrower\noindent
      Keywords and Phrases: #1\par}\fi}

\title [Formal Kuranishi parametrization via homological perturbations] 
{The formal Kuranishi parametrization via the universal
 homological perturbation theory solution of the deformation equation}

\author{Johannes Huebschmann  }
\address{
\noindent
USTL, UFR de Math\'ematiques\\
CNRS-UMR 8524
\\
Labex CEMPI (ANR-11-LABX-0007-01)
\\
\newline
59655 Villeneuve d'Ascq Cedex, France\\
Johannes.Huebschmann@math.univ-lille1.fr
 }

\date{\today}
\numberwithin{equation}{section}

\begin{document}
\setcounter{page}{1}

\maketitle
\medskip
\centerline
{To Tornike Kadeishvili}
\medskip

\begin{abstract} Using homological perturbation theory, 
we develop a formal version of the miniversal deformation
associated with a deformation problem 
controlled by a differential graded Lie algebra
over a field of characteristic zero.
Our approach includes a formal version of the Kuranishi method
in the theory of deformations of complex manifolds.
\end{abstract}

\MSC 

\noindent
Primary: 
14D15 

\noindent
Secondary: 
13D10, 
14B07, 
14B12, 
16S80, 
32G05, 
32G08, 
32S60, 
58K60  

\EndMSC

\KEY Deformation theory, deformation controlled by a differential graded
Lie algebra, miniversal deformation, formal deformation, Kuranishi method

 \EndKEY
{\tableofcontents}

\section{Introduction} 
The Kuranishi map (see Section \ref{fi} below) 
provides 
a parametrization of a neighborhood of a point of the base of a 
Kodaira-Spencer deformation
by a neighborhood of $0$
in the tangent space to the base at that point.
In \cite{MR1932522}, 
prompted by \cite{MR1609624}, J. Stasheff and I developed
a universal solution for the deformation equation,
referred to there as \lq\lq master equation\rq\rq;
that universal solution is
phrased as a Lie algebra valued twisting cochain.
In Theorem \ref{6.7.2} below I show that
that solution recovers a formal version of
the Kuranishi map parametrization.
Lurking behind this observation is, for a general deformation
problem controlled by some differential graded Lie algebra,
 a general formal construction of 
the miniversal deformation.
I shall explain this elsewhere.
Suffice it to mention here that
according to a paradigm adopted by Deligne, Drinfeld, Quillen and others,
any deformation theory problem in characteristic zero
is controlled by a differential graded
Lie algebra, unique up  to homotopy equivalence of differential graded
Lie algebras,
in fact,
up to $L_{\infty}$-equivalence.
See also \cite[Section 5]{MR1710565}.
For a detailed account of deformation theory, see, e.g., \cite{MR981617}.
In the present paper, the various uses of the attribute \lq\lq formal\rq\rq\ 
 (formal miniversal deformation, formal geometry, formal map, formal power
series, formal Lie algebra) 
all reduce to the same mathematical notion, that is, to that of 
a morphism between certain cocommutative coalgebras. 

It is a pleasure to dedicate this paper to Tornike Kadeishvili.
In \cite{MR1710565} I pointed out that there is an intimate
relationship between Berikashvili's functor $\DDD$ and
deformation theory. In particular,
cf. \cite[Section 5]{MR1710565},
there is a striking similarity between
Berikashvili's functor $\DDD$
and a certain functor written 
in the literature
as $\mathrm{Def}_\gg$ for a differential graded Lie
algebra $\gg$.
Here I explain a small aspect of that relationship.
Also, working out the connections with
\cite{MR517083,schlstas} would be an exceedingly attractive project.

\section{Conventions}
The base field  $\fiel$ has characteristic zero.
Following a suggestion of J. C. Moore,
we take graded objects to be {\em externally graded\/}
\cite[VI.2]{maclaboo}, that is, 
rather than taking direct sums,
we work with the homogeneous constituents themselves.
An $\fiel$-chain complex is a $\ZZ$-graded $\fiel$-vector space
$C=\{C_j\}_{j \in \ZZ}$ together 
with a homogeneous square zero operator $d$ of degree $-1$.
We denote the de Rham functor by $\AAA$.
We denote by $s$ the suspension operator in the category of
$\fiel$-chain complexes. Thus, 
for an $\fiel$-chain complex 
$C=\{C_j\}_{j \in \ZZ}$,
the suspended chain complex $sC$ has  $(sC)_{j+1} = C_j$ ($j \in \ZZ$),
and the identity $ds +sd=0$ characterizes the differential on $sC$.
We denote the counit of an $\fiel$-coalgebra
$C$ by $\varepsilon \colon C \to \fiel$,
and we use the same notation  for the identity morphism on an object 
as for the object itself.
Given a map $f \colon X \to Y$,
we use the standard notation $f|_Z$ for the restriction of $f$ to 
a subobject $Z$ of $X$.
We denote the symmetric $\fiel$-algebra on an $\fiel$-vector space $W$ by
$\Sigm[W]$.
We freely use standard homological perturbation terminology
such as contractions etc. The reader can find details in
\cite{MR1109665, MR1932522, MR2640649,  MR2762544, MR2762538,  MR2820385}.
Among the classical references are
\cite{MR0056295, MR0065162, MR0220273, MR0301736, MR662761}.
We denote by $\CCC[\gg]$
the classifying coalgebra or, equivalently,
Cartan-Chevalley-Eilenberg coalgebra, associated with a differential graded
Lie algebra $\gg$.

\section{Some formal geometry}
\label{formalgeom}
We use the term \lq\lq formal geometry\rq\rq\ 
in the sense of \cite{MR0266195, MR0287566, MR2062626}.
That
use relates to that in the sense of 
Grothendieck's
Bourbaki talk in May 1959
 \cite{MR1603467} and to that of his school, see, e.g., \cite{hartsboo} 
and the references there.
Working out the precise relationship would be an attractive endeavor.
Suffice it to mention here the following: The operation of completion
leads to the notion of formal scheme 
\cite[II.9 p.~ 190]{hartsboo};
Grothendieck's Theorem on formal functions 
\cite[Th\'eor\`eme (4.1.5)]{MR0217085},
\cite[III Theorem 11.1 p.~ 277]{hartsboo}
involves this operation of completion;
and the operation of completion admits a description in terms of 
an associated symmetric coalgebra, the basic object on which the 
notion of formal function we use below relies.

\subsection{Symmetric coalgebra}
\label{symmco}
We take the symmetric $\fiel$-coalgebra
functor
$\Sigmc$ 
from 
the category of $\fiel$-vector spaces to that of
cocomplete coaugmented cocommutative 
$\fiel$-coalgebras 
to be the
right adjoint to the
obvious forgetful functor from 
the category of
cocomplete coaugmented cocommutative 
$\fiel$-coalgebras to that of $\fiel$-vector spaces.
We comment on the terminology in Remark \ref{cofree} below.

Let $U$ be an $\fiel$-vector space. 
We denote by 
$\pi_U\colon \Sigmc[U] \to U$
the evaluation at $U$ of the counit
of the adjunction, and we refer to $\pi_U$ as the
{\em cogenerating morphism\/} of $\Sigmc[U]$;
the functor  $\Sigmc$ being right adjoint to the
forgetful function spelled out above
expresses the universal property of
the symmetric $\fiel$-coalgebra  $\Sigmc[U]$
cogenerated by $U$, more precisely,
the universal property of the cogenerating morphism of $\Sigmc[U]$.
The symmetric $\fiel$-coalgebra $\Sigmc[U]$
arises from 
the corresponding
 cocomplete coaugmented cocommutative graded  $\fiel$-coalgebra
$\{\Sigmc_{j}[U]\}_{j \in \mathbb N}$.
With the standard conventions
 $\Sigmc_{0}[U] = \fiel$ and  $\Sigmc_{1}[U] =U$,
and we refer, for $ j \geq 0$, to $\Sigmc_{j}[U]$ as the 
 $j$-th {\em symmetric copower\/} of $U$.
For $\ell \geq 0$, let
$\Sigmc_{\leq \ell}[U] \subseteq \Sigmc[U]$
denote the coaugmentation filtration degree $\ell$  
constituent
of $\Sigmc[U]$;
it acquires itself a 
cocomplete coaugmented cocommutative $\fiel$-coalgebra 
structure in such a way that the 
composite of the injection $\Sigmc_{\leq \ell}[U] \subseteq 
\Sigmc[U]$
with the cogenerating morphism of $\Sigmc[U]$
yields
a cogenerating morphism for $\Sigmc_{\leq \ell}[U]$.

The operation of addition on $U$ induces a multiplication map
on  $\Sigmc[U]$ that turns  $\Sigmc[U]$ into a
commutative and cocommutative $\fiel$-bialgebra,
and multiplication by $-1$ on $U$
yields an antipode such that $\Sigmc[U]$
becomes an  $\fiel$-Hopf algebra.
Since the ground field $\fiel$ has characteristic zero,
the canonical
morphism 
\begin{equation}
\mathrm{can} \colon \Sigm[U] \longrightarrow \Sigmc[U]
\label{canh}
\end{equation}
defined on the symmetric 
$\fiel$-algebra $\Sigm[U]$, endowed with its classical $\fiel$-Hopf
algebra structure,
is an isomorphism
of $\fiel$-Hopf algebras.
The Hopf algebra
$\Sigmc[U]$
is canonically isomorphic to the divided power Hopf algebra generated by $U$.
For reference in Subsections 
\ref{formalfunc} and
\ref{heuristic} below, 
we recall some of the details:

Let $B$ denote an $\fiel$-basis of $U$.
The choice of $B$ determines an isomorphism
from the polynomial algebra $\fiel [B]$ on $B$ onto the
symmetric algebra $\Sigm[U]$.
The 
divided power Hopf algebra $\Gamma[B]$ on $B$
is the commutative and cocommutative $\fiel$-Hopf algebra
having algebra generators $\gamma_j(b)$ ($j \geq 1$), as $b$ 
ranges over $B$,
subject to the relations
$\gamma_j(b) \gamma_k(b) =\tbinom {j+k}j \gamma_{j+k}(b)$
($j,k \geq 0$), where $\gamma_0(b)=1$. 
The formula $\Delta (\gamma_n(b))= \sum _{j+k=n} 
\gamma_j(b) \otimes\gamma_k(b)$ ($n \geq 0$)
characterizes the diagonal map of $\Gamma[B]$.
The composite of the canonical  $\fiel$-Hopf algebra isomorphism 
$\Gamma[B] \to \Sigm[U]$
with \eqref{canh}
is an isomorphism of $\fiel$-Hopf algebras.
Hence the
$\fiel$-coalgebra that underlies $\Gamma[B]$
yields a somewhat more concrete realization of $\Sigmc[U]$.
We postpone more details
to the appendix.

\subsection{Formal functions}
\label{formalfunc}
We maintain the notation $U$ for an $\fiel$-vector space. The 
 $\fiel$-algebra $\Hom(\Sigmc[U],\fiel)$ dual
to the symmetric coalgebra $\Sigmc[U]$ cogenerated by $U$
is the $\fiel$-algebra of {\em formal functions\/}
on $U$ or, equivalently, 
the completion of the symmetric $\fiel$-algebra $\Sigm[U^*]$
on the $\fiel$-dual $U^*$ with respect to the filtration
of $\Sigm[U^*]$
associated with the canonical augmentation map
$\varepsilon \colon \Sigm[U^*] \to \fiel$ of $\Sigm[U^*]$.
The coaugmentation  map $\eta \colon \fiel \to 
 \Sigmc[U]$ of $\Sigmc[U]$
induces
an augmentation map 
of the $\fiel$-algebra $\Hom(\Sigmc[U],\fiel)$, and
we also denote this augmentation map 
by $\varepsilon \colon \Hom(\Sigmc[U],\fiel) \to \fiel$.
For intelligibility we note that,
relative to a  basis of $U$, 
the symmetric $\fiel$-algebra $\Sigm[U^*]$
on $U^*$ comes down to the polynomial algebra in the corresponding
coordinate functions
and the $\fiel$-algebra of formal functions
 to the algebra of formal power series in the
coordinate functions:
Suppose for simplicity that $U$ is finite-dimensional,
choose  an $\fiel$-basis $b_1, \dots, b_n$ of
$U$, and 
write $z_j$ for the
corresponding coordinate functions on $U$
determined by $\langle z_j,b_k\rangle = \delta_{j,k}$ ($1 \leq j,k \leq n$).
A 
formal power series $\sum a_{\mathbf j} z_1^{j_1} \dots z_n^{j_n}$
in the variables
$z_1,\dots,z_n$
with coefficients $a_{\mathbf j} \in \fiel$
determines the 
$\fiel$-linear map
\begin{equation}
\Sigmc[U] \longrightarrow \fiel,\ 
\gamma_{j_1}(b_1)\cdot \ldots \cdot \dots \gamma_{j_n}(b_n) \longmapsto a_{\mathbf j},
\end{equation}
and every such $\fiel$-linear map arises in this way from a unique
formal power series.

\subsection{Formal maps}
Let $U$ and $W$ be $\fiel$-vector spaces, and let $\ell \geq 0$.
Below we use the notation
$(f^{[0]}, \ldots, f^{[\ell]})\colon \Sigmc_{\leq \ell}[U] \to W$
for a linear map $\Sigmc_{\leq \ell}[U] \to W$
having homogeneous constituents
$f^{[j]}\colon \Sigmc_j[U] \to W$.

The assignment to $v \in U$ of
$(1,v, v \otimes v, \ldots, v^{\otimes \ell}) \in \Sigmc_{\leq \ell}[U]$
yields a map
from  $U$ to  $\Sigmc_\ell[U] $, and we denote this map by
$\mathrm{can}\colon U \to  \Sigmc_\ell[U] $.
We define an
{\em algebraic map\/} $f \colon U \to W$ 
{\em of degree\/} $\leq \ell$
to be a map which arises as the composite
\begin{equation}
\xymatrixcolsep{4.5pc}
\xymatrix{
U \ar[r]^{\mathrm{can}}
& \Sigmc_{\leq \ell}[U] 
\ar[r]^{\phantom{aa}(f^{[0]}, \ldots, f^{[\ell]})} 
&
W
}
\end{equation}
of $\mathrm{can}\colon U \to \Sigmc_{\leq \ell}[U]$ with a linear map
$(f^{[0]}, \ldots, f^{[\ell]})\colon \Sigmc_{\leq \ell}[U] \to W$.
An
{\em algebraic map\/} from $U$ to $W$ 
is then a map which is an algebraic map of some degree
$\leq \ell$.

\begin{rema}{\rm
The terminology \lq\lq algebraic map\rq\rq\ is that in
\cite[\S 80C]{MR892316};
that in
\cite {MR0301078}, \cite[\S 9]{MR0065162}
is \lq\lq map of finite degree\rq\rq;
the descriptions in these references
do not involve the symmetric coalgebra, however.
}
\end{rema}

Choose an $\fiel$-basis $b_1,\dots,b_j, \ldots $ of
$U$ and 
write $z_j$ for the
corresponding coordinate function on $U$
determined by $\langle z_j,b_k\rangle = \delta_{j,k}$.
A map $f \colon U \to W$ 
is an algebraic map of degree $\leq \ell$
if and only if, for some $s \leq \ell$,
there is a unique $W$-valued polynomial
$q$ in $s$ variables
of degree $\leq \ell$
such that
$f(z_1 b_1 + \ldots + z_sb_s)= q(z_1, \ldots, z_s)$.
The derivative and the higher derivatives of an algebraic map
make sense algebraically, and hence
the Taylor coefficients 
$\tfrac 1 {j!} f^{(j)}_o$ of an algebraic map 
$f$ at the origin are available.
For $ j \geq 1$,
under the 
degree $j$ constituent 
\begin{equation}
\Sigm^j[U] \longrightarrow \Sigmc_{j}[U],
\end{equation}
of the canonical isomorphism \eqref{canh},
the Taylor coefficient 
$\tfrac 1 {j!} f^{(j)}_o$ of an algebraic map $f$ at the origin 
goes to the constituent $f^{[j]}$.

We define a {\em formal map\/}
from $U$ to $W$ to be
an $\fiel$-linear map 
$\Sigmc[U] \to W$.
An algebraic map 
from $U$ to $W$ 
is, then, a formal map $\Sigmc[U] \to W$
that is zero on  all but finitely many homogeneous constituents
of $\Sigmc[U]$.
Relative to a basis of $U$,
a formal map 
from $U$ to $W$
has the form of a formal power series
in the coordinate functions dual to that basis, with coefficients
from $W$, and the algebraic
functions then amount to the
formal power series that are actually algebraic.

\subsection{Coaffine coalgebras}

Recall that an {\em affine\/} $\fiel$-algebra is a finitely generated 
commutative $\fiel$-algebra, cf., e.g., \cite{hartsboo, MR2977456}. 
In other words, 
an affine $\fiel$-algebra can be written as a quotient algebra
of the symmetric $\fiel$-algebra $\Sigm[W]$ on a
finite-dimensional $\fiel$-vector space $W$.
In particular, an affine $\fiel$-algebra
$A$
isomorphic to $\Sigm[W]$ characterizes the affine
$\fiel$-space having $W$ as its associated vector space,
and an explicit isomorphism between $A$ and
$\Sigm[W]$ then amounts to introducing a coordinate system on $W$.
More generally, an affine $\fiel$-algebra $A$
characterizes an algebraic set, and
an isomorphism $A \cong \Sigm[W]/I$
for some ideal $I$ yields a coordinate ring
for that algebraic set.
Furthermore, an epimorphism or, equivalently, an augmentation map
$A \to \fiel$ corresponds to a point of that algebraic set.
For example, the standard augmentation map
$\Sigm[W] \to \fiel$ corresponds to the origin of $W$.

Accordingly,
we define an $\fiel$-coalgebra $C$ to be {\em coaffine\/}
if, for some  $\fiel$-vector space $W$ of finite dimension,
the coalgebra $C$ embeds,  as an $\fiel$-coalgebra,
into the symmetric $\fiel$-coalgebra
$\Sigmc[W]$ cogenerated by $W$.
Such an embedding yields the analogue of a coordinate ring.
For a coaffine coalgebra $C$,
a coaugmentation map $\eta \colon \fiel \to C$
is the analogue of a point.

Consider a finite-dimensional $\fiel$-vector space $W$.
The symmetric $\fiel$-coalgebra
$\Sigmc[W]$ cogenerated by $W$
has a single coaugmentation, the canonical coaugmentation map
$\eta \colon \fiel \to \Sigmc[W]$.
Indeed, 
the canonical coaugmentation map
$\eta \colon \fiel \to \Sigmc[W]$
induces the canonical augmentation map
$\varepsilon \colon \Hom(\Sigmc[W], \fiel) \to \fiel$,
and this is the only augmentation map of $\Hom(\Sigmc[W], \fiel)$.
The algebra 
$\fiel[W]=\Sigm[W^*]$ is the coordinate algebra $\fiel[W]=\Sigm[W^*]$ 
of $W$ and admits a canonical augmentation map
 $\varepsilon \colon \Sigm[W^*] \to \fiel$ , and
$\Hom(\Sigmc[W], \fiel)$ is canonically isomorphic to the completion
of $\fiel[W]$
relative to the augmentation filtration
with respect to $\varepsilon$. After a choice of basis $b_1, \ldots, b_n$ of $W$,
the algebra $\Sigm[W^*]$ comes down to the polynomial algebra
$\fiel[z_1,\ldots,z_n]$ in the coordinates $z_1,\ldots,z_n$
dual to the basis
and $\Hom(\Sigmc[W], \fiel)$
 to the  algebra
$\fiel[[z_1,\ldots,z_n]]$ of formal power series
in $z_1,\ldots,z_n$.
While a vector $q=(q_1, \ldots, q_n)$ 
of $W$ determines the augmentation map
\[
\varepsilon_q \colon \Sigm[W^*] \longrightarrow \fiel,\ z_j \longmapsto q_j,
\]
since
$\fiel[[z_1,\ldots,z_n]]$ is a (complete) local ring,
the augmentation map $\varepsilon_q$ does not extend
to
$\Hom(\Sigmc[W], \fiel)$ unless 
$(q_1, \ldots, q_n)$ is the origin of $W$. 
Thus a coaffine coalgebra admits at most one point.

\subsection{Formal maps on $\fiel$-varieties}
Consider two finite-dimensional $\fiel$-vector spaces $U$ and $W$
and view $U$ and $W$ as affine $\fiel$-varieties,
cf. \cite{hartsboo, MR2977456}.
An algebraic map $f \colon U \to W$
is then simply a morphism in the category of $\fiel$-varieties.

Consider an affine $\fiel$-variety $\VVV \subseteq U$, let
$A[\VVV]$ denote its coordinate ring, 
recall that $\Sigm[U^*]$ is the coordinate ring of $U$,
write $A[\VVV]$ as $A[\VVV]\cong \Sigm[U^*]/I_\VVV$,
choose a finite set of generators of
the vanishing ideal  $I_\VVV$ of $\VVV$,
let $V$ denote the $\fiel$-vector space
whose dual $V^*$ 
has those generators as its basis
and, accordingly, write 
$A[\VVV]$ 
as the cokernel
of the resulting $\fiel$-linear map
$\vartheta \colon \Sigm[U^*] \otimes V^* \to \Sigm[U^*]$
so that, relative to the obvious
$\Sigm[U^*]$-module structures on  $\Sigm[U^*] \otimes V^*$ and $\Sigm[U^*]$,
\begin{equation}
\Sigm[U^*] \otimes V^* \stackrel{\vartheta}\longrightarrow \Sigm[U^*]\longrightarrow A[\VVV]
\longrightarrow 0
\label{alg}
\end{equation}
is an exact sequence in the category of  $\Sigm[U^*]$-modules.
Since $U$ and $V$ are finite-dimensional,
the morphism
$\vartheta \colon \Sigm[U^*] \otimes V^* \to \Sigm[U^*]$
of $\Sigm[U^*]$-modules arises,
with respect 
to the obvious
$\Sigmc[U]$-comodule structures on  $\Sigmc[U] \otimes V$ and $\Sigmc[U]$,
 from a uniquely determined morphism
$\psi \colon \Sigmc[U] \to \Sigmc[U] \otimes V$
of $\Sigmc[U]$-coalgebras.
Indeed, the chosen finite set $y_1, \ldots, y_\ell \in \Sigm[U^*]$
 of generators of $I_\VVV$ induces an $\fiel$-linear map 
$y \colon \Sigmc[U] \to V= \fiel^{\ell}$, and $\psi$ arises as the composite
\begin{equation*}
\xymatrixcolsep{3.5pc}
\xymatrix{
 \Sigmc[U] \ar[r]^{\Delta\phantom{aaa}}
&\Sigmc[U]  \otimes \Sigmc[U]
\ar[r]^{\phantom{aa} \Sigmc[U]  \otimes y} 
&\Sigmc[U]  \otimes V .
}
\end{equation*}
Define the $\Sigmc[U]$-comodule 
$C[\VVV]$ by requiring that
\begin{equation}
\Sigmc[U] \otimes V \stackrel{\psi}\longleftarrow \Sigmc[U]
\longleftarrow C[\VVV]
\longleftarrow 0
\label{coalg}
\end{equation}
be an exact sequence of
$\Sigmc[U]$-comodules.
The $\Sigmc[U]$-comodule structure of  $C[\VVV]$
passes to a cocommutative $\fiel$-coalgebra structure
on $C[\VVV]$. The resulting $\fiel$-coalgebra
$C[\VVV]$ is manifestly coaffine.
We refer to 
$C[\VVV]$ as the {\em coordinate coalgebra\/}
of $\VVV$.

Suppose that the origin $0$ of $U$ belongs to $\VVV$.
Then the coaugmentation map ${\eta \colon \fiel \to \Sigmc[U]}$
factors through a coaugmentation map
${\eta_\VVV \colon \fiel \to C[\VVV]}$, and the $\fiel$-algebra
$\Hom (C[\VVV],\fiel)$
is canonically isomorphic to the completion $\widehat \OOO_0$
of the local ring  $\OOO_0$ of $0$ on $\VVV$, cf., e.g.,
\cite[I.3 p. 16]{hartsboo}, \cite[p. 71]{MR2977456}.
(The hypothesis in \cite{hartsboo} that
the ground field be algebraically closed is irrelevant at this point.
In standard algebraic geometry,
the algebraic closedness of the ground field
is needed to ensure that the familiar association
to a variety of its annihilation ideal
 induces  
not only an injection into but in fact a
bijection onto
the family of radical ideals.)
The subalgebra 
of $\Hom (C[\VVV],\fiel)$ that consists of the $\fiel$-linear
maps  arising as restrictions of
$\fiel$-linear
maps $\Sigmc[U] \to \fiel$
which are non-zero on at most
finitely many homogeneous constituents
$\Sigmc_j [U]$ ($j \geq 0$)
of $\Sigmc[U]$
recovers the coordinate ring
$A[\VVV]$ of $\VVV$.
Since the origin of $U$ belongs to $\VVV$, the standard augmentation map
of $\Sigm[U^*]$ factors through an augmentation map
${\varepsilon_\VVV \colon A[\VVV] \to \fiel}$,
and the $\fiel$-algebra $\Hom(\Sigmc[U],\fiel)\cong\widehat \OOO_0$
is canonically isomorphic to the completion 
of $A[\VVV]$ relative to the augmentation filtration.
The affine $\fiel$-variety $\VVV$
being considered as an abstract variety
(i.e., independently of the embedding into $U$),
by introducing suitable
coordinates we can, of course, arrange for any point of
$\VVV$ to be the origin
with respect to an embedding of $\VVV$ into $U$.

We define a {\em formal map\/}
$f \colon \VVV \to W$ to be 
an $\fiel$-linear map 
$C[\VVV] \to W$.
Let $\WWW \subseteq W$ be an $\fiel$-variety.
More generally,
we define a {\em formal map\/}
$f \colon \VVV \to \WWW$ to be 
a formal map
$f \colon \VVV \to W$,
i.e.,
an $\fiel$-linear map 
$C[\VVV] \to W$,
that has the property that
the values of the induced morphism
$C[\VVV] \to \Sigmc[W]$
of $\fiel$-coalgebras lie in
$C[\WWW]\subseteq\Sigmc[W]$. 

It is manifest that any formal map $f \colon \VVV \to \WWW$ 
arises as
the restriction
of an $\fiel$-linear map
$\Sigmc[U] \to W$. 
A formal map $f\colon \VVV \to \WWW$ corresponds to an
 ordinary morphism $\VVV \to \WWW$ 
of algebraic $\fiel$-varieties
if and only if $f$ arises  from
an $\fiel$-linear map
$\Sigmc[U] \to W$
which is non-zero on at most
finitely many homogeneous constituents
$\Sigmc_j [U]$ ($j \geq 0$) of $\Sigmc[U]$.

\section{Generalities on differential 
graded Lie algebras and deformations}
\label{general}

\subsection{Setting}

Consider a differential graded $\fiel$-Lie algebra $\gg$
that is non-positive in the sense that
$\gg_j$ is zero for $ j >0$.

\begin{examp}[Kodaira-Spencer Lie algebra] 
\label{KS}
{\rm See \cite{MR0112157, MR0112154}. Take $\fiel$ to be the field $\CC$
of complex numbers,  
consider a complex manifold $M$, 
let $\tau_M$
denote the holomorphic tangent bundle
of $M$, let $\overline \partial$ be the corresponding
Dolbeault operator,
and let
$\gg = (\AAA^{(0,*)}(M,\tau_M), \overline \partial)$
be the {\em Kodaira-Spencer algebra\/} of $M$,
endowed with the homological grading 
\begin{equation}
\gg_0 = \AAA^{(0,0)}(M,\tau_M),
\quad
\gg_{-1} = \AAA^{(0,1)}(M,\tau_M),
\quad
\gg_{-2} = \AAA^{(0,2)}(M,\tau_M),
\quad
\text{etc.}
\end{equation}
Thus, with our convention on degrees,
$\HH_*(\gg) = \HH^{-*}(M,\tau_M)$,
the
cohomology of $M$ with values in the sheaf of germs of
holomorphic vector fields.
}
\end{examp}

\begin{rema} 
{\rm The deformation theory of twilled Lie-Rinehart algebras
\cite{twilled, MR1764437}
leads to a generalization of the Kodaira-Spencer algebra.
Details are yet to be worked out.
}
\end{rema}

A similar differential graded Lie algebra also controls
deformations of other mathematical objects such as 
those of algebras, of complex manifolds,
of germs of complex hypersurfaces,
of singularities,
of representations of discrete groups,
of flat vector bundles,
of commutative algebras,
of rational homotopy types \cite{MR517083,schlstas},
etc.

Let
\begin{equation}
\Nsddata {\gg} {\phantom a\nabla}{\phantom a\pi}
{\HH(\gg)}h 
\label{6.0}
\end{equation}
be a contraction of $\fiel$-chain complexes, and let
$
\HHH = \ker (h) \cap \ker(d) = \nabla \HH(\gg)$.
Then,
for $j\leq 0$, the degree $j$ constituent $\gg_j$ decomposes as
\begin{equation}
\gg_{0} = \HHH_{0} \oplus h(d\gg_{0}),
\quad \gg_j = d \gg_{j+1} \oplus \HHH_j \oplus h(d\gg_j),\ \text{for}\ j<0.
\label{4.5.1}
\end{equation}
In the situation of Example \ref{KS},
\eqref{4.5.1} plays the role of a {\em Hodge decomposition\/}.
On p.~19 
of \cite{MR0195995},
Nijenhuis and Richardson indeed
refer to a decomposition of the kind \eqref{4.5.1}
(not using the language of homological perturbation theory)
as a \lq\lq Hodge decomposition\rq\rq.

\subsection{Universal solution of the deformation equation
\cite[Theorem 2.7]{MR1932522}}

For ease of exposition, we recall that solution;
see also \cite[Theorem 1]{MR2762544} for a 
detailed proof of \cite[Theorem 2.7]{MR1932522}.

Consider
the graded symmetric coalgebra 
$\Sigmc[s\HH(\gg)]$ on $s\HH(\gg)$
which underlies
the classifying coalgebra
$\CCC [\HH(\gg)]$ 
of the graded Lie algebra $\HH(\gg)$, 
let $\tau_{\HH(\gg)} \colon \Sigmc[s\HH(\gg)] \to \HH(\gg)$
be the universal twisting cochain of $\HH(\gg)$,
and let $\tau_1$ denote the composite
\begin{equation*}
\tau_1 
\colon \Sigmc[s\HH(\gg)] \stackrel{\tau_{\HH(\gg)}}\longrightarrow
\HH(\gg)
\stackrel{\nabla}\longrightarrow\gg;
\end{equation*}
for consistency of the exposition, 
let $\DDD_1=0 \colon \Sigmc[s\HH(\gg)] \to \Sigmc[s\HH(\gg)]$,
and denote the ordinary differential of 
the classifying coalgebra
$\CCC [\HH(\gg)]$
of $\HH(\gg)$
(Cartan-Chevalley-Eilenberg differential)
by $\DDD^1\colon \Sigmc[s\HH(\gg)] \to \Sigmc[s\HH(\gg)]$.

The universal twisting cochain
$\tau_{\HH(\gg)}$ of $\HH(\gg)$
and hence
the morphism
$\tau_1$,
both defined on $\Sigmc[s\HH(\gg)]$,
are non-zero only on the homogeneous degree $1$ constituent
$\Sigmc_1[s\HH(\gg)] (= s\HH(\gg))$ of $\Sigmc[s\HH(\gg)]$.
Define the family 
$\{\tau^\ell \colon\Sigmc_\ell[s\HH(\gg)] \to\gg\}_{\ell \geq 2}$ 
of $\fiel$-linear maps recursively by
\begin{alignat}{1}
\tau^\ell &= \frac 12 h([\tau^1,\tau^{b-1}] +  \dots + [\tau^{b-1},\tau^1])
\label{3.1.4}
\end{alignat}
and, for $\ell \geq 2$, let
$\DDD^{\ell-1}$
be
the coderivation 
of $\Sigmc[s\HH(\gg)]$
that the identity
\begin{equation}
\tau_{\HH(\gg)} \DDD^{\ell-1} =
\frac 12 \pi
([\tau^1,\tau^{\ell-1}] +  \dots + [\tau^{\ell-1},\tau^1])
\colon
\Sigmc_\ell[s\HH(\gg)]
\to \HH (\gg)
\label{3.1.6}
\end{equation}
characterizes.

Since the values of $\tau_1 = \tau^1$ lie in the cycles of
$\gg$, for $\ell=2$, 
the formula \eqref{3.1.6} reproduces the ordinary differential
of the classifying coalgebra of $\HH(\gg)$,
that is, 
the operator $\DDD^1$ coincides with that differential,
and there is no conflict of notation.
Recall that, for a filtered chain complex $X$, a
{\em perturbation\/} of the differential
$d$ on $X$ is a homogeneous operator $\partial$ on $X$
of degree $-1$ that lowers filtration such that the operator
$d+\partial$ has square zero, i.e., is itself a differential on $X$.
For later refence, we spell out the following;
see \cite[Theorem 2.7]{MR1932522} for more details.

\begin{prop}
\label{prop1}
\begin{enumerate}
\item[{\rm (i)}]
The infinite sum
\begin{equation}
\DDD= \DDD^1 + \DDD^2 + \dots 
\label{3.1.5}
\end{equation}
converges (naively, that is, when the sum is 
applied to an element, only finitely many terms are non-zero)
 and yields a
coalgebra differential
on $\Sigmc[s\HH(\gg)]$
so that,
relative to the coaugmentation filtration,
 $\DDD^2 + \dots $
is a perturbation of
the Lie algebra homology operator
$\DDD^1$ on
$\CCC [\HH(\gg)]$; the infinite sum
$\tau^1 + \tau^2 + \dots$ 
converges and yields a Lie algebra
 twisting cochain 
\begin{equation}
\tau =\tau_1 + \tau^2 +  \dots\
\colon \Sigmc_{\DDD}[s\HH(\gg)] \to\gg
\label{ltc}
\end{equation}
such that 
\begin{align}
\pi \tau&=\tau_{\HH(\gg)}\colon \Sigm^{\mathrm c}[s\HH(\gg)] \longrightarrow \HH(\gg),
\label{twist33}
\\
h \tau &= 0; \label{twist44}
\end{align}
and $\tau$ and $\DDD$ are natural in the contraction {\rm \eqref{6.0}}.

\item[{\rm (ii)}] The adjoint
$\overline \tau \colon \Sigm_{\ppartial}^{\mathrm c}[s\HH(\gg)]\to
\mathcal C[\gg]$ 
of the Lie algebra twisting cochain 
$\tau \colon \Sigmc_{\DDD}[s\HH(\gg)] 
\to \CCC [\gg]$, necessarily a morphism of
coaugmented differential graded coalgebras,
extends to a contraction
\begin{equation}
   \left(\Sigm_{\ppartial}^{\mathrm c}[s\HH(\gg)]
     \begin{CD}
      \null @>{\overline \tau}>> \null\\[-3.2ex]
      \null @<<{\Pi}< \null
     \end{CD}
    \mathcal C[\gg], H\right)
\label{3.1.18}
  \end{equation}
of chain complexes which is natural in terms of the data. 

\item[{\rm (iii)}]
The adjoint
$\overline \tau \colon \Sigmc_{\DDD}[s\HH(\gg)] 
\to \CCC [\gg]$ of the Lie algebra twisting cochain $\tau$
induces  an isomorphism on homology.

\item[{\rm (iv)}]
If  $\pi[\tau^j,\tau^k] = 0$ for $j+k>2$, that is, if
$[\tau^j,\tau^k] \colon 
\Sigmc[s\HH(\gg)] \to\gg$
does not hit the summand $\HHH$ for $j+k>2$, 
the operators $\DDD^\ell$ are zero for $\ell \geq 2$,
that is, the differential graded Lie algebra
$\gg$ is formal in the sense that it is
$L_\infty$-equivalent to its homology Lie algebra. \qed
\end{enumerate}
\end{prop}

The Lie algebra twisting cochain \eqref{ltc},
viewed as an element of the differential graded Lie algebra
$\LLL = \Hom(\Sigmc_{\DDD}[s\HH(\gg)],\gg)$, 
is 
the universal solution of the deformation equation constructed
in \cite[Theorem 2.7]{MR1932522}.

\subsection{Postnikov tower}

For $k \leq 0$, let 
$\gg^{(k)}$ be the differential graded Lie ideal in $\gg$
which has 
$\gg^{(k)}_j = 0$ for
$k< j \leq 0$,
$\gg^{(k)}_k =h(d\gg_k)$, and
$\gg^{(k)}_j = \gg_j$ for
$j<k$, and consider the differential graded 
quotient Lie algebra
\begin{equation*}
\gg(k) = \gg \big / \gg^{(k)}.
\end{equation*}
As a chain complex,
$\gg(k)$
has 
$\gg(k)_j = \gg_j$ for
$k<j\leq 0$ and
$\gg(k)_k =d \gg_{k+1} \oplus \HHH_k$.
As $k$ varies, this construction yields a tower
\begin{equation*}
\cdots \longrightarrow \gg(-k-1) \longrightarrow \gg(-k) \longrightarrow \cdots \longrightarrow
\gg(-2) \longrightarrow \gg(-1) \longrightarrow \gg(0)=\HHH_0 =\HH_0(\gg)
\end{equation*}
of differential graded Lie algebras,
and $\gg$ is canonically isomorphic to the projective limit of this tower.
Since in ordinary rational homotopy theory, a tower
of this kind recovers a Postnikov system, we refer to this tower
as the {\em Postnikov tower\/} of $\gg$.

\subsection{Associated reduced differential graded Lie algebra}
We say that a differential graded Lie algebra $\gg$ is {\em reduced\/}
when $\gg_j$ is trivial for $ j \geq 0$.
To any differential graded Lie algebra $\gg$,
the following procedure
assigns a reduced
differential graded Lie algebra 
${\widetilde {\gg}= (\widetilde {\gg}_{-1},\widetilde {\gg}_{-2}) }$,
cf., e.g., \cite{MR1065894},
(but this observation was known before, e.g., to R. Hain).
In degree $-1$, 
let $\widetilde {\gg}_{-1}$
be a complement
in
$\gg_{-1}$
of the boundaries $d\gg_0 \subseteq \gg_{-1}$
and, for $j \leq -2$, let
$\widetilde {\gg}_j = \gg_j$.
Furthermore, define the differential
 and the graded
bracket on $\widetilde {\gg}$
by restriction.
This yields a differential graded Lie algebra
$\widetilde {\gg}$; the degree $-1$
constituent $\widetilde {\gg}_{-1} $
depends also on the choice of complement of the boundaries $d \gg_0$.

\subsection{Differential graded Lie algebras concentrated in degrees $-1$ and $-2$}

For any non-positive differential graded Lie algebra $\gg$,
together with a contraction of the kind \eqref{6.0},
setting $\kk = \widetilde \gg(-2)$
we obtain
a differential graded Lie algebra 
concentrated only in degrees
$-1$ and $-2$.

Let $\kk$ be a 
differential graded Lie 
algebra over $\fiel$ concentrated only in degrees
$-1$ and $-2$.
Let
$V = s \kk$ be its suspension in the category of $\fiel$-chain complexes,
and let $d$ denote the 
differential on $V$ that arises as the
suspension of the differential of $\kk$.
Thus $V$ is
concentrated
in degrees $0$ and $-1$.
Let
\begin{equation}
\Nsddata {\kk} {\phantom a\nabla}{\phantom a\pi}
{\HH(\kk)}h 
\label{6.1}
\end{equation}
be a contraction,
write $\vv = s\HH(\kk)$
and, with an abuse of notation, let
\begin{equation}
\Nsddata {V} {\phantom a\nabla}{\phantom a\pi}{\vv}h 
\label{6.2}
\end{equation}
be the induced contraction.
Thus
$\vv_0 = \ker(d)$ and
$\vv_{-1} = \mathrm {coker}(d)$, 
and
$A_0=hdV_0$ 
is a complement 
of
$\vv_0$ in $V_0$
and 
$B_{-1}=dV_0$ one of
$\vv_{-1}$ in $V_{-1}$ such that
the restriction of $d$ to $A_0$ is a linear isomorphism
from $A_0$ onto $B_{-1}$ whose inverse is given by $h$. 
Further, the composite
$hd$ is the projection from $V_0$ onto $A_0$, and
the composite
$dh$ is the projection from $V_{-1}$ onto $B_{-1}$.

The graded Lie bracket of $\kk$ corresponds to a homogeneous 
quadratic map
$q \colon V_0 \to V_{-1}$
and is in fact determined by such a map.
Here $q$ being homogeneous quadratic means that
there is a linear map $Q \colon 
\Sigmc_2[V_0] \to V_{-1}$
such that $q$ is the composite
\begin{equation*}
V_0 \longrightarrow \Sigmc_2[V_0] \stackrel{Q} \longrightarrow V_{-1},
\end{equation*}
the unlabelled arrow being the canonical map from
$V_0$ into its second symmetric copower
$\Sigmc_2[V_0]$.
With a slight abuse of notation,
we write the resulting map
from $\Sigmc[V_0]$ to $V_{-1}$
which arises the composite
with the projection
from $\Sigmc[V_0]$ to 
$\Sigmc_2[V_0]$
still as
\begin{equation*}
Q \colon 
\Sigmc[V_0] \stackrel{\mathrm{proj}}
\longrightarrow \Sigmc_2[V_0] \stackrel{Q} \longrightarrow V_{-1}.
\end{equation*}
Relative to the decomposition
$V_{-1} =B_{-1} \oplus \vv_{-1}$,
the quadratic map $q$ and the linear map $Q$ have the form
\begin{equation} q=(q_B,q_\vv) \colon V_0 \longrightarrow V_{-1},\quad 
Q =(Q_B,Q_\vv)\colon \Sigmc[V_0] \longrightarrow V_{-1}.
\end{equation}
With reference to a basis $b_1,b_2,\dots$ of
$V_0$ and corresponding coordinate functions
$z_j$ determined by $\langle z_j,b_k\rangle = \delta_{j,k}$,
we can write the quadratic map $q$ in the form
\begin{equation*}
q(z_{k_1} b_{k_1} + \dots + z_{k_n} b_{k_n}) 
= \sum a_{k_j} z_{k_j}^2  + \sum a_{\alpha,\beta} z_{k_\alpha} z_{k_\beta},
\end{equation*}
for suitable coefficients
$a_{k_j},a_{\alpha,\beta} \in V_{-1}$.

The differential $d$ on $V$
induces a coalgebra differential 
on the graded symmetric
coalgebra 
$\Sigmc[V]$ and, with an abuse of notation,
we write this differential as $d$.
As a graded coalgebra,
$\Sigmc[V] = \Sigmc[V_0] \otimes \Lc[V_{-1}]$,
the tensor product of 
the
ordinary symmetric coalgebra $\Sigmc[V_0]$
on the degree zero constituent $V_0$ of $V$ 
with 
the
ordinary exterior coalgebra $\Lc[V_{-1}]$
on the degree $-1$ constituent $V_{-1}$ of $V$.
We note that, by construction, the coalgebra
$\Sigmc[V_0]$ is concentrated in homological degree zero.
The contraction \eqref{6.2}
induces
a coalgebra contraction
\begin{equation}
\Nsddata {(\Sigmc[V],d)} {\phantom{a}\Sigmc[\nabla]\phantom{a}}
{\phantom{a}\Sigmc[\pi]\phantom{a}}
{\Sigmc[\vv]}{\Sigmc[h]}; 
\label{6.3}
\end{equation}
here the induced morphisms 
$\Sigmc[\nabla]$ and $\Sigmc[\pi]$
are morphisms of differential graded coalgebras,
the graded symmetric
coalgebra 
$\Sigmc[\vv]$
has zero differential,
and the operator
$\Sigmc[h]$
arises from the corresponding operator
$\TTTc[h]$
on the differential graded tensor coalgebra
$\TTTc[V]$
by symmetrization.
In particular,
$\Sigmc[\nabla]$
induces an isomorphism
$\Sigmc[\vv] \to \HH_*(\Sigmc[V],d)$
of coalgebras 
or, equivalently,
$\Sigmc[\pi]$
induces an isomorphism
$\HH_*(\Sigmc[V],d)\to \Sigmc[\vv]$
of coalgebras.

Requiring that the diagram
\begin{equation*}
\CD
\Sigmc[V_0]
@>{\partial}>> \Sigmc[V]
\\
@V{Q}VV
@VV{\mathrm {can proj}}V
\\
V_{-1} @>{\mathrm {inclusion}}>> V
\endCD
\end{equation*}
be commutative
determines a coalgebra perturbation
\begin{equation}
\partial \colon \Sigmc[V] \longrightarrow \Sigmc[V]
\label{6.4}
\end{equation} 
of the coalgebra differential $d$ on $\Sigmc[V]$,
and the perturbed differential graded coalgebra
$(\Sigmc[V],d+\partial)$
yields 
the classifying coalgebra $\mathcal C[\kk]$
of $\kk$.

In the same vein, the components $Q_B$ and $Q_\vv$ 
of $Q$ determine
coalgebra perturbations
$\partial_B, \partial_\vv \colon \Sigmc[V] \to \Sigmc[V]$
of the coalgebra differential $d$ on  $\Sigmc[V]$,
and $\partial= \partial_B+ \partial_\vv$.
Let 
$\kk_{\mathrm a}$ be the 
differential graded Lie algebra
having the same underlying  chain complex as $\kk$
but having Lie bracket determined by
$Q_B$ only, that is,
the component 
into $\HHH_{-2}= \nabla(\HH(\kk))= s^{-1}(\vv)$
of the graded
bracket is zero.
Equivalently,
requiring that
the classifying coalgebra $\mathcal C[\kk_{\mathrm a}]$
of $\kk_{\mathrm a}$ be
$(\Sigmc[V],d+\partial_B)$ 
determines $\kk_{\mathrm a}$.

\subsection{Applying the universal solution of the deformation equation}

Let $\DDD\colon \Sigmc[\vv] \to \Sigmc[\vv]$
denote the coalgebra differential
\eqref{3.1.5}
on $\Sigmc[\vv]$ 
which results 
from applying the construction in Proposition \ref{prop1}(i) 
to the contraction 
\eqref{6.1} and perturbation $\partial= \partial_B+ \partial_\vv$.

\begin{prop}
\label{6.5}
\begin{enumerate}
\item[{\rm (i)}]
Applying
the construction in Proposition {\rm \ref{prop1}}{{\rm (i)}}
to the contraction {\rm \eqref{6.1}\/} and
perturbation
$\partial_B$ 
of the coalgebra differential $d$ of $\Sigmc[V]$
yields the zero coalgebra differential
on $\Sigmc[\vv]$
and,
furthermore,
a twisting cochain
\begin{equation}
\tau\colon \Sigmc[\vv] \longrightarrow \kk_{\mathrm a} .
\label{6.5.1}
\end{equation}

\item[{\rm (ii)}]
As a morphism of degree $-1$ from
$\Sigmc[\vv]$ to $\kk_{\mathrm a} = \kk$
(viewed as a graded vector space),
{\rm \eqref{6.5.1}}
coincides
with the twisting cochain
\begin{equation}
\Sigm_{\DDD}^{\mathrm {c}}[\vv] \longrightarrow \kk
\label{6.5.2}
\end{equation}
resulting from
applying
the construction in Proposition  {\rm \ref{prop1}} {{\rm (i)}}
to the contraction {\rm \eqref{6.1}\/} and
perturbation
$\partial= \partial_B+ \partial_\vv$.
\end{enumerate}
\end{prop}

\begin{proof}
This
follows
at once from the recursive description \eqref{3.1.6}
of the terms of the perturbed differential on
$\Sigmc[\vv]$
since the values of the graded Lie bracket of
$\kk_{\mathrm a}$ are boundaries
and from the 
recursive description \eqref{3.1.4}
of the two twisting cochains
under consideration
since
the homotopy operators on
$\kk$ and $\kk_{\mathrm a}$ 
(written as $h$)
are the same. 
\end{proof}

Thus perturbing the
coalgebra differential $d$ on $\Sigmc[V]$
by means of $\partial_B$
yields a degree $-1$ 
morphism $\tau\colon \Sigmc[\vv] \longrightarrow \kk$ 
that is, 
relative to the zero differential on 
$\Sigmc[\vv]$,
a twisting cochain with values in 
$\kk_{\mathrm a}$,
and perturbing further
by means of the operator $\partial_\vv$
yields the in general 
non-trivial coalgebra differential 
$\DDD$ on $\Sigmc[\vv]$
but does not change 
the twisting cochain $\tau$ in the sense that
$\tau$ is as well a twisting cochain 
$\Sigmc_\DDD[\vv] \longrightarrow \kk$. 
In particular,
the graded Lie algebra $\HH(\kk_{\mathrm a})$ is abelian.

\begin{prop}
\label{6.6}
The differential graded Lie algebra $\kk_{\mathrm a}$ 
is formal in the sense that it is
$L_\infty$-equivalent to its homology Lie algebra.
 Moreover, 
besides the adjoint
$\overline \tau\colon
\Sigmc[\vv]
\to
\Sigm_{d+\partial_B}^{\mathrm {c}}[V]
=\mathcal C[\kk_{\mathrm a}]
$ of $\tau$ 
being a morphism of differential graded coalgebras
inducing an isomorphism on homology,
\begin{equation}
\Sigmc[\pi]\colon
\mathcal C[\kk_{\mathrm a}]=\Sigm_{d+\partial_B}^{\mathrm {c}}[V]
\longrightarrow
\Sigmc[\vv]
\end{equation}
is a morphism of differential graded coalgebras as well 
which induces as isomorphism on homology,
$\Sigmc[\vv]$ being endowed with the zero differential
and hence being equal to its homology.
Furthermore,
$\Sigmc[\pi] \circ \overline \tau 
= \mathrm {Id}_{\Sigmc[\vv]}$.
Finally, the injection 
$\overline \tau$
is compatible with the perturbations $\DDD$ and $\partial$ as well,
that is to say,  is a morphism
\begin{equation}
\overline \tau\colon
\Sigm_{\DDD}^{\mathrm {c}}[\vv]
\longrightarrow
\Sigm_{d+\partial}^{\mathrm {c}}[V]
=\mathcal C[\kk]
\end{equation}
of differential graded coalgebras as well.
\end{prop}

\begin{proof}
In the corresponding
contraction 
$\Nsddata {\kk_{\mathrm a}} {\nabla}{\pi}
{\HH(\kk_{\mathrm a})}h$,
cf. \eqref{6.1}, the projection $\pi$ is compatible with the Lie brackets.
This implies the first two assertions.
Inspection shows that
$\overline \tau$
is compatible with the perturbations $\DDD$ and $\partial$ as well.
Furthermore,
the composite
\begin{equation*}
\tau_{\HH(\kk_{\mathrm a})} \circ \Sigmc[\pi]
\circ \overline \tau \colon
\Sigmc[\vv] \longrightarrow \HH(\kk_{\mathrm a})
\end{equation*}
equals
$\tau_{\HH(\kk_{\mathrm a})}$
whence
$\Sigmc[\pi]\circ \overline \tau$
is the identity of
$\Sigmc[\vv]$
and thence
$\overline \tau$ is a section for
$\Sigmc[\pi]$. 
In particular, the corresponding contraction \eqref{3.1.18}
for
$\kk_{\mathrm a}$
has the form
\begin{equation}
\Nsddata {\Sigm_{d+\partial_B}^{\mathrm {c}}[V]}
{\overline \tau}
{\phantom{a} \Sigmc[\pi] \phantom{a}}
{\Sigmc[\vv]}{h},  
\label{6.6.1}
\end{equation}
with the additional property that
$\Sigmc[\pi]$
is a morphism of differential graded coalgebras as well.
\end{proof}

\section{Formal inverse of the Kuranishi projection}
\label{fi}

As before, $\kk$ denotes a differential graded
$\fiel$-Lie algebra concentrated in degrees $-1$ and $-2$,
the notation $V$ refers to the suspension $s\kk$ of $\kk$
as a chain complex, and 
we argue in terms of a contraction of the kind \eqref {6.2}.
Thus
$\vv_0 = \ker(d)$,
$\vv_{-1} = \mathrm {coker}(d)$,
$A_0=hdV_0$, $B_{-1}=dV_0$, so that
\begin{equation*}
V_0= \vv_0 \oplus A_0, \
V_{-1} =B_{-1} \oplus\vv_{-1}
\end{equation*}
in such a way that
the restriction of $d$ to $A_0$ is a linear isomorphism
from $A_0$ onto $B_{-1}$, 
and the restriction of $h$ to $B_{-1}$ yields
the inverse thereof.
Further, the composite
$hd$ is the projection from $V_0$ onto $A_0$, and
the composite
$dh$ is the projection from $V_{-1}$ onto $B_{-1}$.

\subsection{Kuranishi map} \label{6.7}
Consider the map 
\begin{equation}
J=d+q =(d+q_B, q_\vv)\colon V_0 \longrightarrow 
V_{-1} = B_{-1} \times  \vv_{-1}.
\label{kur1}
\end{equation}
Let $\MMM_\kk \subseteq V_0$ be the 
non-singular $\fiel$-variety that consists 
of all $x=(x_1,x_2) \in V_0= \vv_0 \oplus A_0$
satisfying the equation $x_2 + h(q_B(x))=0$.
The origin $0$ of  $V_0$  lies in $\MMM_\kk$. 
Since, for $x=(x_1,x_2) \in V_0= \vv_0 \oplus A_0$,
\begin{equation*}
(d+q_B)(x_1,x_2)= d(x_2) +q_B(x_1,x_2) \in B_{-1}
\end{equation*}
is zero if and only if
\[
 hd(x_2) +hq_B(x_1,x_2) =  x_2 +hq_B(x_1,x_2)
\]
 is zero,
the zero locus
$\VVV_\kk=J^{-1}(0) \subseteq V_0$,
an $\fiel$-subvariety of
$V_0$, not necessarily non-singular,
lies in $\MMM_\kk$
as the subvariety
\begin{equation}
\VVV_\kk =(q_\vv|_{\MMM_\kk})^{-1}(0) \subseteq \MMM_\kk .
\end{equation}

The expression 
\begin{equation}
F(x_1,x_2) = (x_1, x_2) + h (q_B(x_1,x_2)),\quad x_1\in \vv_0,\, x_2 \in A_0.
\label{6.7.1}
\end{equation}
characterizes an algebraic map $F \colon V_0  \to V_0$,
i.e., a morphism of $\fiel$-varieties,
the equation $x_2 = 0$ characterizes
the (algebraic) tangent space 
$\TTT_0\MMM_\kk$ to $\MMM_\kk$ at $0$
and, since the projection
$V_0 \to \vv_0$ is an algebraic map,
 the restriction of $F$
to $\MMM_\kk$ 
is the  projection from $\MMM_\kk$ onto the tangent space 
$\TTT_0\MMM_\kk=\vv_0\subseteq V_0$ to $\MMM_\kk$ at $0$,
necessarily an algebraic map.

We recall the classical situation \cite{MR0141139, MR0355111},
see also
\cite{MR606458, 
MR506229,  
MR1065894, 
MR1363857, MR1369463, MR1376296}. Now
the base field $\fiel$ is that of the reals or that of the complex numbers,
$\MMM_\kk$
is a smooth manifold, and $F$ is the ordinary
{\em Kuranishi map\/}.
In an open neighborhood $\mathcal N$ of $0$
in the tangent space $\TTT_0\MMM_\kk=\vv_0$,
the inverse $\iota \colon \mathcal N \to\MMM_\kk$
of the projection to the tangent space at the origin
then exists as an analytic map and
parametrizes a neighborhood $\mathcal U$
of  $0$ in $\MMM_\kk$ in the sense that
$\iota \colon \mathcal N \to\mathcal U$
is an analytic isomorphism
onto $\mathcal U \subseteq \MMM_\kk$.
With the notation
$\KKK_\kk =
\mathrm{pr}(\mathcal U\cap J^{-1}(0)) \subseteq \vv_0$ 
for the image
in $\vv_0$ 
of the intersection $\mathcal U\cap J^{-1}(0) \subseteq \MMM_\kk$,
the restriction of
$\iota$ yields, furthermore, a homeomorphism
\begin{equation}
\iota|_{\KKK_\kk} \colon \KKK_\kk \longrightarrow
\VVV_\kk ,
\label{homeo}
\end{equation}
indeed, an isomorphism of (germs of) analytic spaces.
The reader can find more details on analytic spaces in
 \cite{MR0217337}, \cite{MR755331} (where the terminology is 
\lq\lq complex spaces\rq\rq) and in the references there. 
Extending terminology in 
\cite{MR1046568},
we refer to $\KKK_\kk$ 
as the {\em Kuranishi space\/}
associated with the data.
In the situation of the Kodaira-Spencer Lie algebra,
see Example \ref{KS} above, 
the main result of \cite{MR0141139}, see also
\cite{MR0355111}, says that
the Kuranishi space is the base for the miniversal deformation
of the complex structure
of the underlying smooth real manifold.

\subsection{Formal inverse}
Now we show that, 
over a general field $\fiel$ of characteristic zero,
for a general differential graded $\fiel$-Lie algebra of 
the kind $\kk$ under discussion, 
the inverse of the projection from the non-singular $\fiel$-variety
$\MMM_\kk$
to its tangent space 
at the origin
still
makes sense on the formal level.

The Kuranishi map is
manifestly an algebraic map 
and hence a formal map from  $V_0$ to $V_0$.
Consequently the projection from $\MMM_\kk$
to its tangent space 
at the origin
is an algebraic map 
and hence a formal map.

The homology coalgebra
$\HH_0(\Sigmc_{d+\partial_B}[V])\subseteq \Sigmc[V_0]$
is canonically isomorphic to the coordinate coalgebra
$C[\MMM_\kk]$ of the non-singular
$\fiel$-variety
$\MMM_\kk \subseteq V_0$,
and the restriction
of the projection 
$\Sigmc[\pi]\colon
\Sigmc[V_0]
\longrightarrow
\Sigmc[\vv_0]$ to
$C[\MMM_\kk]\subseteq \Sigmc[V_0]$, that is, the composite
\begin{equation}
\Sigmc[\pi]_*\colon C[\MMM_\kk]\subseteq \Sigmc[V_0]\stackrel{\Sigmc[\pi]}\longrightarrow
\Sigmc[\vv_0],
\label{pro}
\end{equation}
is the formal map associated with the projection from
$\MMM_\kk$ to its tangent space $\vv_0$ at the origin.
This composite is indeed the morphism of coalgebras
that the morphism $\Sigmc[\pi]$ of differential graded coalgebras
induces on degree zero homology,
whence the notation $\Sigmc[\pi]_*$.
Likewise,
the homology coalgebra
$\HH_0(\Sigmc_{d+\partial}[V])\subseteq \Sigmc[V_0]$
is canonically isomorphic to the coordinate coalgebra
$C[\VVV_\kk]$ of the
$\fiel$-variety
$\VVV_\kk \subseteq\MMM_\kk \subseteq V_0$ in such a way that
the 
injection
$C[\VVV_\kk]\subseteq \Sigmc[V_0]$
of the coordinate coalgebra $C[\VVV_\kk]$
of $\VVV_\kk$ into $\Sigmc[V_0]$
factors as
$C[\VVV_\kk]\subseteq C[\MMM_\kk]\subseteq \Sigmc[V_0]$.
We define the {\em Kuranishi coalgebra\/}
associated with $\kk$ and the contraction \eqref{6.1}
to be the $\fiel$-coalgebra
$C_\kk= \HH_0(\Sigmc_{\DDD}[\vv])\subseteq \Sigmc[\vv_0]$.
The Kuranishi coalgebra
associated with $\kk$ and the contraction \eqref{6.1}
is the formal analogue of the Kuranishi space $\KKK_\kk$,
cf. Subsection \ref{6.7} above.

\begin{thm}
\label {6.7.2}The constituent
\begin{equation}
\overline \tau\colon
\Sigmc[\vv_0]
\longrightarrow
C[\MMM_\kk]
\subseteq \Sigmc[V_0]
\label{adj}
\end{equation}
of the adjoint $\overline \tau$
of the twisting cochain 
$\tau \colon \Sigmc[\vv]
\to \kk$, cf.
\eqref{6.5.1},
is the 
inverse of \eqref{pro}, and
\eqref{adj} and \eqref{pro}
are mutually inverse isomorphisms
of coalgebras.
In other words:
The
morphism
{\rm \eqref{adj}} of coalgebras
is the formal inverse 
of the projection from
$\MMM_\kk$ to its tangent space $\vv_0$ at the origin.
This formal inverse restricts to an isomorphism
\begin{equation}
\overline \tau\colon
C_\kk
\longrightarrow
C[\VVV_\kk]
\subseteq \Sigmc[V_0]
\label{adj2}
\end{equation}
of $\fiel$-coalgebras.
\end{thm}

\begin{proof}
This results form the contraction 
\eqref{6.6.1} or, equivalently, from Proposition \ref{6.6}.
\end{proof}

The $\fiel$-coalgebra isomorphism 
\eqref{adj2}
is the formal analogue of \eqref{homeo} above.
We shall explain elsewhere
how we can interpret the universal solution  $\tau$ 
of the deformation equation
as
a {\em formal miniversal deformation\/}.

\subsection{Heuristic interpretation}
\label{heuristic}
We can somewhat loosely interpret the definition 
of the Kuranishi coalgebra
by saying that
the Kuranishi coalgebra  $C_\kk$ of $\kk$
defines a {\em formal\/} $\fiel$-variety 
$\KKK_\kk \subseteq \vv_0$
having $C_\kk$ as its coordinate coalgebra $C[\KKK_\kk]$.
Indeed, in a sense,
the operator $\DDD$
yields equations defining
the 
formal Kuranishi space
$\KKK_{\kk}$
as a formal variety in $\vv_0$. 

Somewhat more explicitly: 
Suppose that $\vv_0$ and
$\vv_{-1}$ are finite-dimensional,
and let $m = \dim \vv_{-1}$.
The graded symmetric coalgebra
$\Sigmc[\vv]$ has
\begin{equation}
(\Sigmc[\vv])_{-j} = \Sigmc[\vv_0] \otimes
\Lambda^j [\vv_{-1}],
\quad 0 \leq j\leq m,
\label{6.7.3}
\end{equation}
whence
$(\Sigmc[\vv])_0 = \Sigmc[\vv_0]$
and
$(\Sigmc[\vv])_{-1} 
= \Sigmc[\vv_0] \otimes \vv_{-1}$.
Choose  an $\fiel$-basis $b_1,\dots,b_n$ of
$\vv_0$ and 
write $z_j$ for the
corresponding coordinate function on $V_0$
determined by $\langle z_j,b_k\rangle = \delta_{j,k}$ ($1 \leq j,k \leq n$).
As noted in 
Subsection \ref{formalfunc} above, the 
formal power series 
in the variables
$z_1,\dots,z_n$
correspond bijectively to the formal functions on $\vv_0$.
Since
$\pi \overline \tau = \mathrm {Id}$,
the adjoint $\overline \tau
\colon \Sigmc[\vv]
\to  \Sigmc[V]$ 
of $\tau$
is injective.

By construction, the restriction, to
$\Sigmc[\vv_0]$, of the
composite $s \circ \tau \colon \Sigmc[\vv]
\stackrel{\tau}\to  \kk \stackrel{s}\to V$
then takes the form
\begin{equation}
(s \circ \tau)(z_1 b_1 + \dots + z_n b_n) 
= \sum a_{\mathbf j} z_1^{j_1} \dots z_n^{j_n},
\label{6.7.4}
\end{equation}
the right-hand side
of this equation being
a formal power series
having coefficients
$a_{\mathbf j} \in V_0$.
After a choice of basis $\beta_1,\dots,\beta_m$ of $\vv_{-1}$,
the 
composite 
\begin{equation}
\pi_{\DDD}\colon (\Sigmc[\vv])_0 = \Sigmc[\vv_0]
\longrightarrow \vv_{-1}
\end{equation}
of the operator 
\begin{equation}
\DDD\colon
(\Sigmc[\vv])_0 = \Sigmc[\vv_0]
\longrightarrow
(\Sigmc[\vv])_{-1} = \Sigmc[\vv_0] \otimes
\vv_{-1}
\end{equation} 
with the projection to
$\vv_{-1}$
amounts to
$m$ formal power series 
$\Phi_1,\dots,\Phi_m$
in the variables 
$z_1,\dots,z_n$.
As an $\Sigmc[\vv_0]$-comodule,
the
Kuranishi coalgebra 
$C_\kk=\HH_0(\Sigm_{\DDD}^{\mathrm {c}}[\vv])$
is isomorphic to the kernel of the map
\begin{equation}
\xymatrixcolsep{4pc}
\xymatrix{
\Sigmc[\vv_0] \ar[r]^{\Delta\phantom{aaaa}} 
&\Sigmc[\vv_0] \otimes \Sigmc[\vv_0]
\ar[r]^{\phantom{aa}\Sigmc[\vv_0] \otimes \pi_\DDD} 
&
 \Sigmc[\vv_0] \otimes \vv_1 .
}
\end{equation}

\subsection{Classical case}
\label{6.7.5} Over $\fiel = \RR$
or $\fiel = \mathbb C$,
under appropriate circumstances,
by the Banach implicit function theorem,
in a neighborhood of zero,
the formal power series \eqref{6.7.4} converges
and yields the inverse 
of the 
projection from $\MMM_\kk$
to its tangent space
near $0$; this inverse is
a map
from a neighborhood of $0$ in
$\vv_0$ to a neighborhood of $0$ in $\MMM_\kk$
in the usual sense.
This observation applies, e.g., to the Kodaira-Spencer algebra, cf. 
Example \ref{KS} above.

\subsection{Illustration} 
\label{6.7.6} Consider the  
plane $\fiel^2$ with coordinates $v$ and $u$.
Let $\kk$ be the differential graded Lie algebra
concentrated in degrees $-1$ and $-2$
having $\kk_{-1}= \fiel^2$ and $\kk_{-2}= \fiel$,
with differential $d\colon \kk_{-1}\to \kk_{-2}$
the projection to the second copy of $\fiel$ in $\kk_{-1}$,
that is, $d(v,u)=u$,
and with Lie bracket
\begin{equation}
\bra \colon \kk_{-1} \otimes \kk_{-1} \longrightarrow  \kk_{-2}= \fiel,
\ 
[(v_1,u_1),(v_2,u_2)]= \tfrac 12 (u_1u_2 +v_1v_2).
\end{equation}
For simplicity, we now identify in notation
$\kk$ and its suspension $V$.
Then the identity
$q_B(v,u) = u^2 + v^2$ characterizes the
 corresponding quadratic map 
\[
q_B\colon \kk_{-1}=V_0 = \fiel^2 \longrightarrow 
V_{-1} = B_{-1} = \kk_{-2} = \fiel.
\] 
Thus $\vv_0 = \ker(d)$ is the $v$-axis,
$\vv_{-1} = \mathrm {coker}(d)$ is trivial,
and the map \eqref{kur1} above comes down to the
real function $J=d+q_B$
on $V_0= \fiel^2$ given by
$J(v,u) = u + u^2 + v^2$.

The identity $h(u)=(0,u)$ characterizes
the corresponding  homotopy $h \colon V_{-1} =B_{-1} \to V_0$,
the constituent of $h$
on $V_0$ being zero for degree reasons.
The non-singular variety (smooth manifold when $\fiel = \RR$) 
$\MMM_\kk$ given by the equation $J(u,v)=0$ 
is then the circle
\begin{equation}
\left(u+\frac 12\right)^2 + v^2 = \frac 14 ,
\end{equation}
the tangent space
$\TTT_0\MMM_\kk=\vv_0$
to $\MMM_\kk$ at the origin
is the $v$-axis, and
the Kuranishi map \eqref{6.7.1} now has the form
\begin{equation}
F \colon \fiel^2 \longrightarrow \fiel^2,\ F(v, u)=(v,u) +
h(q_B(v,u))=(v,u+ u^2 + v^2).
\end{equation}
Visibly, the restriction of $F$ to $\MMM_\kk$ is the projection
to  the $v$-axis, i.e., to 
the tangent space
$\TTT_0\MMM_\kk=\vv_0$
to $\MMM_\kk$ at the origin.
Over $\fiel = \RR$,
the classical recursion for
the non-zero Taylor coefficients $a_2$, $a_4$, etc., of
$u$ as 
an analytic function
of the variable $v$ close to the origin of $V_0= \RR^2$
reads
\begin{equation}
\begin{aligned}
u&=a_2v^2 +a_4 v^4 + a_6 v^6 + \ldots
\\
a_4 v^4 + a_6 v^6 + \ldots +(a_2v^2 +a_4 v^4 + a_6 v^6 + \ldots)^2 &=0
\\
a_2+1&=0
\\
a_4+a^2_2&=0
\\
a_6+2a_2a_4&=0
\\
a_8 +2 a_2a_6 + a_4^2 &= 0,
\end{aligned}
\label{recurs}
\end{equation}
etc., this recursion
yields the 
familiar expression 
\begin{equation}
u= \sqrt{\frac 14-v^2} - \frac 12=
- v^2 - v^4  - 2 v^6 -\dots , 
\label{analf}
\end{equation}
and the assignment to $v$ of $(v,u(v))$
parametrizes
 the circle under discussion
near the origin of $V_0$.

The power series on the right-hand side of \eqref{analf},
viewed as a formal power series,
is a special case of \eqref{6.7.4} above.
Examining 
how the homological perturbation theory construction
of the twisting cochain \eqref{ltc} 
leads to that power series 
is instructive:
Write $b$ for the standard basis element
$b=(1,0)$ of the $v$-axis $\vv_0$. In terms of this notation,
$\HH_{-1}(\kk)= \fiel \langle s^{-1} (b) \rangle$,
the $1$-dimensional $\fiel$-vector space
having the desuspension $s^{-1}(b)$ of $b$ as its basis.
By construction
\begin{align*}
\tau_{\HH(\kk)}(b)&= s^{-1}(b),
\\
s  \tau_1(b)\,(\,= s \nabla \tau_{\HH(\kk)}(b)) &= b = (1,0) \in V_0,
\end{align*}
that is, the composite $s\circ \tau_1 \colon \vv_0 \to V_0$
coincides with the corresponding injection
of the kind
$\nabla \colon \vv_0 \to V_0$ in the contraction \eqref{6.2}.
The recursion \eqref{3.1.4} translates to the recursion
\eqref{recurs}
for the non-zero Taylor coefficients $a_2$, $a_4$, etc. of $u$ as an analytic 
function of the variable $v$.
The recursion \eqref{3.1.4} entails that
the non-zero Taylor coefficients $a_2$, $a_4$, etc. of $u$
are members of $\fiel$ (rather than of a suitable extension field).

\begin{rema}
{\rm The field $\fiel(v)$ of rational functions in the variable
$v$ is the function field of the $v$-axis, viewed 
as an affine $\fiel$-variety,
the field $\fiel(\MMM_\kk)$
of rational functions on $\MMM_\kk$
is the  Galois extension of degree $2$ that arises by
adjoining, to $\fiel(v)$, the roots $u$ of the equation $x^2+v^2+x=0$, 
and the projection $\MMM_\kk \to \vv_0$ corresponds to the field extension
$\fiel(v) \subseteq \fiel(\MMM_\kk)$.
This kind of observation is, of course, valid under the general 
circumstances in Subsection \ref {6.7} above:
The projection $\MMM_\kk \to \vv_0$
determines a Galois field extension
$\fiel(\vv_0) \subseteq \fiel(\MMM_\kk)$
for the associated fields
of rational functions.
This observation implies that the formal
Kuranishi map 
cannot yield an ordinary morphism of
$\fiel$-varieties
 that is globally defined.
It remains to be seen to what extent this 
is relevant for deformation theory.
}
\end{rema}

\section*{Appendix}
Let $U$ be an $\fiel$-vector space. 

\subsection*{Symmetric coalgebra}
For $j \geq 2$,  
the  $j$-th symmetric copower 
 $\Sigmc_{j}[U] \subseteq U^{\otimes j}$
of $U$ is
the linear $\fiel$-subspace
of invariants under the canonical action,
on $U^{\otimes j}$,
of the symmetric group on $j$ letters.

The graded
symmetric $\fiel$-coalgebra 
cogenerated by $U$
is the graded coaugmented subcoalgebra of the graded $\fiel$-tensor coalgebra
$\{ U^{\otimes j}\}_{j \in \mathbb N}$ cogenerated by $U$
having  $\fiel$ as its homogeneous degree $0$
and $U$ as its homogeneous degree $1$
 constituents and,
for $j \geq 2$, as
 homogeneous degree $j$ constituent
the $j$-th symmetric copower
 $\Sigmc_{j}[U]$ of $U$.
The $\fiel$-vector space $U$
being concentrated in degree $1$,
the canonical projection from
$\{ \Sigmc_j[U]\}_{j \in \mathbb N}$ to $U$
yields the requisite cogenerating morphism
of graded $\fiel$-vector spaces.
Totalization yields
the symmetric $\fiel$-coalgebra  $\Sigmc[U]$
cogenerated by $U$.
For
$\ell \geq 0$, the coaugmentation filtration degree $\ell$  
constituent
$\Sigmc_{\leq \ell}[U] \subseteq \Sigmc[U]$
of $\Sigmc[U]$
has, for $j \leq \ell$, as its homogeneous constituents
the homogeneous  constituents $\Sigmc_{j}[U]$
of $\Sigmc[U]$.

Recall that the canonical diagonal map $U \to U \oplus U$
induces an $\fiel$-bialgebra structure on the symmetric 
$\fiel$-algebra $\Sigm[U]$ and
that multiplication by $-1$ on $U$
yields an antipode such that $\Sigm[U]$
acquires an  $\fiel$-Hopf algebra structure.
The canonical projection $\pi_U\colon \Sigm[U] \to U$
induces, via the universal property of
$\Sigmc[U]$, the canonical 
morphism 
$\Sigm[U] \to \Sigmc[U]$,
cf. \eqref{canh},
of graded $\fiel$-coalgebras.
This canonical morphism, viewed as an $\fiel$-linear map, coincides with 
the $\fiel$-linear map that underlies the canonical 
morphism 
$\Sigm[U] \to \Sigmc[U]$
of graded $\fiel$-algebras
associated with the canonical injection $U \to \Sigmc[U]$
and,
since the ground field $\fiel$ has characteristic zero,
this morphism is an isomorphim
of graded $\fiel$-Hopf algebras.

As in Subsection \ref{symmco} above,
choose an $\fiel$-basis $B$ of $U$.
The canonical morphism
${\Gamma[B]\to \Sigmc[U]}$ of $\fiel$-coalgebras
is an isomorphism, and the canonical map $\mathrm{can}\colon \fiel[B] \to \Gamma[B]$
renders the diagram
\begin{equation*}
\begin{CD}
\fiel[B] @>{\mathrm{can}}>>\Gamma[B]
\\
@VVV
@VVV
\\
\Sigm[U]
@>{\mathrm{can}}>>
\Sigmc[U]
\end{CD}
\end{equation*}
commutative.
The inverse of $\mathrm{can}\colon \fiel[B] \to \Gamma[B]$
sends $\gamma_j(b)$ to $\tfrac 1 {j!} b^j$ ($j \geq 1$)
as $b$ ranges over $B$ and thereby introduces divided powers in 
$\Sigm[U]$.

\begin{rema}\label{cofree}
{\rm
The
obvious forgetful functor from 
the category of general cocommutative
$\fiel$-coalgebras, i.e.,
cocommutative
$\fiel$-coalgebras not necessarily endowed with a coaugmentation,
to that of $\fiel$-vector spaces
admits likewise a right adjoint
\cite[p. 129 Theorem 6.4.3]{MR0252485}. Write this functor as
$\Sigma^{\mathrm c}$. Consider an $\fiel$-vector space $U$.
The $\fiel$-coalgebra $\Sigma^{\mathrm c}[U]$
comes with a canonical coaugmentation map,
and
$\Sigmc[U]$ 
is the largest cocomplete $\fiel$-subcoalgebra
of $\Sigma^{\mathrm c}[U]$ cogenerated by $U$,
but the two coalgebras do not coincide unless 
$U$ is the trivial vector space, cf.
\cite{MR782637, MR1992043}.
For example, when $U$ has dimension $1$ and when $z$ denotes the standard
coordinate function on $U$, the symmetric $\fiel$-coalgebra 
$\Sigmc[U]$ amounts to the coalgebra that underlies the
polynomial algebra $\fiel[z]$ (as noted before), whereas
$\Sigma^{\mathrm c}[U]$ 
comes down to the $\fiel$-coalgebra that underlies the
algebra of rational functions in the variable $z$ which are 
regular at the origin.
For general $U$, in \cite{MR0252485}, see also
\cite[p. 26, p. 51]{MR594993}, \cite[Section 3  Note p. 720]{MR2035107},
the $\fiel$-coalgebra $\Sigma^{\mathrm c}[U]$
is referred to as the symmetric coalgebra on $U$.
}
\end{rema}
\subsection*{Formal tangent space}
Let $C$ be  an $\fiel$-coalgebra $C$, and let $\Coder(C)$
denote the $\fiel$-vector space of coderivations of $C$,
endowed with the Lie bracket arising from the commutator 
of $\fiel$-linear endomorphisms of $C$.
More generally, for a  $C$-bicomodule $B$, we denote
by $\Coder(B,C)$
the $\fiel$-vector space of coderivations from $B$ to $C$.

A morphism $C \to \Sigmc[U]$
of $\fiel$-coalgebras induces an $\Sigmc[U]$-bicomodule structure
on $C$.
Recall that, in view of the universal property
 of the symmetric $\fiel$-coalgebra
$\Sigmc[U]$ on $U$, 
the cogenerating surjection
$\Sigmc[U] \to U$
induces an isomorphism
$\Coder(C,\Sigmc[U]) \to \Hom(C,U)$  of $\fiel$-vector spaces.
In particular,
the cogenerating surjection
$\Sigmc[U] \to U$ of $\fiel$-vector spaces
induces an  $\fiel$-vector space isomorphism
$\Coder(\Sigmc[U]) \to \Hom(\Sigmc[U],U)$.

Now, the coaugmentation map $\eta \colon \fiel \to 
 \Sigmc[U]$ turns the ground field $\fiel$ into an
$\Sigmc[U]$-bicomodule, and we use the notation
 $\fiel_\eta$ for this $\Sigmc[U]$-bicomodule.
The above observation entails that the cogenerating surjection
$\Sigmc[U] \to U$
induces an isomorphism
\[
\Coder(\fiel_\eta,\Sigmc[U]) \longrightarrow \Hom(\fiel,U) \cong U
\]  
of $\fiel$-vector spaces.
Thus the $\fiel$-vector space 
$\Coder(\fiel_\eta,\Sigmc[U])$ is canonically isomorphic to the
$\fiel$-vector space of primitives in $\Sigmc[U]$
relative to the canonical Hopf algebra structure of 
$\Sigmc[U]$. 
We refer to the space of primitives in $\Sigmc[U]$
as the {\em formal tangent space\/} to $U$ at $0$.
When $\fiel$ is the field of real or that of complex numbers,
this formal tangent space comes of course down to  the  ordinary 
tangent space to $U$ at $0$.

Let
$\fiel_\varepsilon$
denote the base field $\fiel$, endowed with
the 
$\Hom(\Sigmc[U],\fiel)$-module structure
induced by the augmentation map
$\varepsilon \colon \Hom(\Sigmc[U],\fiel) \to \fiel$.
The canonical map
\begin{equation}
\Coder(\fiel_\eta,\Sigmc[U]) \longrightarrow 
\Der(\Hom(\Sigmc[U],\fiel), \fiel_\varepsilon)
\label{can1}
\end{equation}
is an isomorphism.
Under \eqref{can1}, a basis element of $U$ 
goes to the associated operation of partial derivative.
This justifies the terminology
\lq\lq formal tangent space\rq\rq\  to $U$ at $0$.

\section*{Acknowledgement}

I am indebted to Jim Stasheff for a number of
most valuable comments on a draft of the paper.
I gratefully acknowledge support by the CNRS and by the
Labex CEMPI (ANR-11-LABX-0007-01).

\def\cprime{$'$} \def\cprime{$'$} \def\cprime{$'$} \def\cprime{$'$}
  \def\cprime{$'$} \def\cprime{$'$} \def\cprime{$'$} \def\cprime{$'$}
  \def\dbar{\leavevmode\hbox to 0pt{\hskip.2ex \accent"16\hss}d}
  \def\cprime{$'$} \def\cprime{$'$} \def\cprime{$'$} \def\cprime{$'$}
  \def\cprime{$'$} \def\Dbar{\leavevmode\lower.6ex\hbox to 0pt{\hskip-.23ex
  \accent"16\hss}D} \def\cftil#1{\ifmmode\setbox7\hbox{$\accent"5E#1$}\else
  \setbox7\hbox{\accent"5E#1}\penalty 10000\relax\fi\raise 1\ht7
  \hbox{\lower1.15ex\hbox to 1\wd7{\hss\accent"7E\hss}}\penalty 10000
  \hskip-1\wd7\penalty 10000\box7}
  \def\cfudot#1{\ifmmode\setbox7\hbox{$\accent"5E#1$}\else
  \setbox7\hbox{\accent"5E#1}\penalty 10000\relax\fi\raise 1\ht7
  \hbox{\raise.1ex\hbox to 1\wd7{\hss.\hss}}\penalty 10000 \hskip-1\wd7\penalty
  10000\box7} \def\polhk#1{\setbox0=\hbox{#1}{\ooalign{\hidewidth
  \lower1.5ex\hbox{`}\hidewidth\crcr\unhbox0}}}
  \def\polhk#1{\setbox0=\hbox{#1}{\ooalign{\hidewidth
  \lower1.5ex\hbox{`}\hidewidth\crcr\unhbox0}}}
  \def\polhk#1{\setbox0=\hbox{#1}{\ooalign{\hidewidth
  \lower1.5ex\hbox{`}\hidewidth\crcr\unhbox0}}}

\end{document}